\documentclass[11pt,a4paper,reqno]{amsart}
\usepackage{amssymb,amsfonts,amsthm,amscd,stmaryrd}
\usepackage{mathrsfs}

\usepackage[font=small,labelfont=bf]{caption}
\usepackage{times,graphicx,epsfig}
\usepackage{colortbl}
\usepackage{esint}
\newcommand{\beq}{\begin{equation}}
	\newcommand{\eeq}{\end{equation}}

\newcommand{\medint}{-\kern -,375cm\int}
\newcommand{\medintinrigo}{-\kern -,315cm\int}

\numberwithin{equation}{section}
\usepackage{mathtools}
\usepackage{esint}

\newcommand{\pedfrac}[2]{\genfrac{}{}{0pt}{}{#1}{#2}}

\theoremstyle{plain}

\newtheorem{teo}{Theorem}[section]
\newtheorem{cor}[teo]{Corollary}
\newtheorem{prop}[teo]{Proposition}
\newtheorem{lemma}[teo]{Lemma}
\theoremstyle{remark}
\newtheorem{oss}[teo]{Remark}
\theoremstyle{definition}
\newtheorem{definition}[teo]{Definition}

\title[A spherical flatness index and a stability inequality]{A spherical flatness index\\ and a stability inequality for  harmonic pseudospheres}

\author[A. Buffoni, G. Cupini, E. Lanconelli]{Andrea Buffoni, Giovanni Cupini,  Ermanno Lanconelli}

\address{Andrea Buffoni, Giovanni Cupini, Ermanno Lanconelli: Dipartimento di Matematica, Universit\`a di
	Bologna\\ Piazza di Porta San Donato 5,
	40126 - Bologna, Italy}
\email{andrea.buffoni4@unibo.it}
\email{giovanni.cupini@unibo.it}
\email{ermanno.lanconelli@unibo.it}

\begin{document} 
\maketitle
% \begin{center}
	%  {\sc Giovanni Cupini - Nicola Fusco - Ermanno Lanconelli - Xiao Zhong} \end{center}
\begin{abstract}
We introduce a new flatness index for the boundary of an open subset  $\Omega$ of $\mathbb{R}^n$, $n\ge 2$. This index provides  a necessary condition for $\partial\Omega$ to be a harmonic pseudosphere and  sufficient conditions for a harmonic pseudosphere to be a Euclidean sphere. These conditions will  follow from a stability inequality formulated in terms of a harmonic invariant, the Kuran gap, recently introduced   by the last two authors.
	\end{abstract} 
\section{Introduction and main results}

\subsection{Harmonic pseudospheres}
Let $\Omega\subseteq\mathbb{R}^n$ be a bounded open set whose boundary has $(n-1)$-Hausdorff measure $|\partial\Omega|<+\infty$. If $x_0\in\Omega$, then $\partial\Omega$ is called a harmonic pseudosphere centered at $x_0$ if $$u(x_0)=\fint_{\partial\Omega} u \, d\sigma$$ for every function $u$ harmonic in $\Omega$ and continuous up to $\partial\Omega$. \\
The classical Gauss mean value theorem for harmonic functions, implies that every Euclidean sphere is a harmonic pseudosphere; the vice versa, however, is not true in general. Indeed, in $\mathbb{R}^n$, $n\geq2$, there exist harmonic pseudospheres, homeomorphic to the unit Euclidean sphere $\partial B(0,1)$ via H\"older-continuous homeomorphisms, which are not Euclidean spheres; see the 1991 paper \cite{LewVog91} by Lewis and Vogel. For earlier results in dimension $n=2$, see the papers \cite{KelLav} and \cite{Sha}.

In the present paper we introduce a new flatness index for $\partial\Omega$ at certain points, which provides a necessary condition for $\partial\Omega$ to be a harmonic pseudosphere and sufficient conditions for a harmonic pseudosphere to be a Euclidean sphere. These conditions follow from a stability inequality formulated in terms of a harmonic invariant, the Kuran gap, introduced in \cite{CupLan25}.

\subsection{The Kuran gap}
Let $\alpha\in\mathbb{R}^n$. We call $\alpha$-Kuran function the map 
$$x\mapsto k_{\alpha}(x):=1+h_{\alpha}(x)$$ 
where $$h_{\alpha}(x):=|\alpha|^{n-2}\frac{|x|^2-|\alpha|^2}{|x-\alpha|^n}, \quad x\in\mathbb{R}^n\backslash\{\alpha\}.$$
This function was introduced by  Kuran in 1972 \cite{Kur}.

The $\alpha$-Kuran function is harmonic in $\mathbb{R}^n\backslash\{\alpha\}$ and vanishes at $x=0$. \\
In \cite{CupLan25}, the Kuran gap of $\partial\Omega$ with respect to $x_0$, denoted by $\mathcal{K}(\partial\Omega,x_0)$ was introduced and defined as follows. If $x_0=0$, then 
\begin{align*}
\mathcal{K}(\partial\Omega,0):&=\sup_{\alpha\notin\overline{\Omega}}\left|k_{\alpha}(0)-\fint_{\partial\Omega}k_{\alpha}(x) \, d\sigma(x)\right| \nonumber \\
&=\sup_{\alpha\notin\overline{\Omega}}\left|\fint_{\partial\Omega}k_{\alpha}(x) \, d\sigma(x)\right|.
\end{align*}
If $x_0\neq0$, then $$\mathcal{K}(\partial\Omega,x_0):=\mathcal{K}(-x_0+\partial\Omega,0).$$
It is easy to recognize that the Kuran gap is invariant with respect to Euclidean translations and dilations. 

Moreover, and crucially for our purposes, $\mathcal{K}(\partial\Omega,x_0)=0$ whenever $\partial\Omega$ is a harmonic pseudosphere centered at $x_0$. \\ 
By the translation invariance of the Kuran gap, it suffices to prove this statement in the case $x_0=0$. In this case, indeed, since $k_{\alpha}$ is harmonic in $\mathbb{R}^n\backslash\{\alpha\}$, if $\partial\Omega$ is a harmonic pseudosphere centered at $x_0$, we have $$k_{\alpha}(x_0)=\fint_{\partial\Omega}k_{\alpha} \, d\sigma \quad \forall \alpha\notin\overline{\Omega}.$$
Hence, by the very definition of Kuran gap, $\mathcal{K}(\partial\Omega,0)=0$.

\subsection{The spherical flatness index}
Let $\Omega\subseteq\mathbb{R}^n$ be open and let $x_0\in\Omega$. Assume that every compact subset of $\partial\Omega$ has finite $(n-1)$-Hausdorff measure. We denote by $T(\partial\Omega,x_0)$ the touching set of $\partial\Omega$ with respect to $x_0$, defined as follows $$T(\partial\Omega,x_0):=\{z\in\partial\Omega \ : \ |x_0-z|=\mathrm{dist}(\partial\Omega,x_0)\}.$$
Equivalently, letting $r:=\mathrm{dist}(\partial\Omega,x_0)$, $$T(\partial\Omega,x_0)=\partial\Omega\cap\partial B(x_0,r).$$
If $z\in T(\partial\Omega,x_0)$ we let $r:=|z-x_0|$ and we define 
\begin{equation}\label{e: definizione spherical index}
S_z(\partial\Omega,x_0):=\frac{2}{n\omega_n r}\liminf_{R\searrow 0}\left(\liminf_{\pedfrac{\alpha\to z}{\alpha\notin\overline{\Omega}}}\int_{\partial \Omega\cap B(z,R)} \frac{\langle\alpha-x,\alpha-x_0\rangle}{|x-\alpha|^n} \, d\sigma(x)\right).
\end{equation}
We call $S_z(\partial\Omega,x_0)$ the \textit{spherical flatness index} of $\partial\Omega$ at $z$.

Before proceeding, we find it convenient to comment on this definition through the following remarks.

\begin{oss}\label{r: poisson kernel}
The map $$(x,\alpha)\mapsto \frac{2}{n\omega_n r}\frac{\langle \alpha-x, \alpha-x_0\rangle}{|x-x_0|^n}$$ is essentially the Poisson kernel of the halfspace $$\Pi_z:=z+\{x\in\mathbb{R}^n \ : \ \langle x-x_0, z-x_0 \rangle \geq0\}.$$
\end{oss}

\begin{oss}\label{r: invarianza indice sferico}
The spherical flatness index is invariant under rigid motions.
\end{oss}

\begin{oss}\label{r: indice uguale ad 1}
The spherical flatness index $S_z(\partial\Omega,x_0)$ equals $1$ for all domains whose boundary is flat in a neighborhood of $z$. Moreover, $S_z(\partial\Omega,x_0)=1$ when $\Omega$ is any Euclidean ball containing $x_0$, not necessarily centered at $x_0$.
\end{oss}

The statements in Remark \ref{r: poisson kernel} and \ref{r: invarianza indice sferico} can be easily proved. Those in Remark \ref{r: indice uguale ad 1}, will be explicitly proved in Section \ref{s: computations}, in Proposition \ref{p: indice sferico nel caso piatto} and Proposition \ref{p: indice sferico nel caso sfera non centrata}.

In Section \ref{s: computations}, we will also prove the following proposition, which relies on the definition below.
First of all, we fix the following notations.

\medbreak
We split $\mathbb{R}^n$ as $\mathbb{R}^{n-1}\times\mathbb{R}$ and if $z=(z_1,\ldots,z_n)\in\mathbb{R}^n$, we denote it as $z=(z',z_n)$ with $z'=(z_1,\ldots,z_{n-1})\in\mathbb{R}^{n-1}$. \\
Given $R>0$ and $-\infty<a<z_n<b<+\infty$, we consider a rectangular neighborhood of the point $z$ of this kind $$U_R(a,b):=\widehat{B}_R(z')\times(a,b)$$ where $\widehat{B}_R(z'):=\left\{y\in\mathbb{R}^{n-1} \ : \ \left|y-z'\right|<R\right\}$. 
\begin{definition}[Lipschitz flatness]
	Let $\Omega\subseteq\mathbb{R}^n$ be open and let $z\in\partial\Omega$. We say that $\partial\Omega$ is \textit{Lipschitz flat} at the point $z$ if there exists a function $$\psi:\widehat{B}_R(z')\rightarrow(a,b)$$ with $R>0$, $-\infty<a<z_n<b<+\infty$ such that:
	\begin{itemize}
		\item[(i)] $\psi$ is Lipschitz, differentiable and $\nabla\psi(z'+t\xi)\to\nabla\psi(z')$ as $t\to0$ 
		for almost every $\xi\in\mathbb{R}^{n-1}$;
		\item[(ii)] $\partial \Omega \cap U_R(a,b)=\left\{(y,\psi(y)) \ : \ y\in\widehat{B}_R(z')\right\}$;
		\item[(iii)] $ \Omega \cap U_R(a,b)=\left\{(y,t) \ : \ y\in\widehat{B}_R(z'), \ a<t<\psi(y)<b\right\}$.
	\end{itemize}
\end{definition}
Then, the following result holds.
\begin{prop}\label{p: indice sferico nel caso lipschitz piatto}
 Let $\Omega\subseteq \mathbb{R}^n$ be a bounded open set and let $x_0\in \Omega$.  If $\partial\Omega$ is Lipschitz flat at a point $z\in T(\partial\Omega,x_0)$, then $$S_z(\partial\Omega,x_0)\leq1.$$
\end{prop}

\subsection{Main results}
The central result of this paper is the following theorem, from which the other main results follow.

\begin{teo}\label{t: relazione KG e S_z}
Let $\Omega\subseteq\mathbb{R}^n$ be a bounded open set such that $|\partial\Omega|<+\infty$. Consider a point $x_0\in\Omega$ and let $z\in T(\partial \Omega,x_0)$. Then, the following stability inequality holds:
\begin{equation}\label{e: relazione KG e S_z}
\mathcal{K}(\partial\Omega,x_0)\geq\frac{|\partial \Omega|-S_z(\partial \Omega,x_0)|\partial B|}{|\partial \Omega|},
\end{equation}
where $B$ is the biggest ball centered at $x_0$ and contained in $\Omega$, i.e. $$B:=B(x_0,r) \quad \mathrm{with} \quad r:=|z-x_0|=\mathrm{dist}(x_0,\partial\Omega).$$
\end{teo}

From this theorem we easily obtain the results of the following corollaries, in which $\Omega$ always denotes a bounded open subset of $\mathbb{R}^n$ with $|\partial\Omega|<+\infty$ and containing the point $x_0$.

\begin{cor}\label{c: condizione necessaria pseudosfere}
If $\partial \Omega$ is a harmonic pseudosphere centered at $x_0$, then 
\begin{equation}\label{e: CN pseudosfere}
S_z(\partial \Omega,x_0)\geq 1 \quad \forall z\in T(\partial \Omega,x_0).
\end{equation}
\end{cor}

\begin{cor}\label{c: indice S e disuguaglianza}
If there exists a point $z\in T(\partial \Omega,x_0)$ such that $$S_z(\partial \Omega,x_0)\leq1$$ then 
\begin{equation}\label{e: maggiorazione Kuran Gap}
\mathcal{K}(\partial \Omega,x_0)\geq\frac{|\partial \Omega|-|\partial B|}{|\partial \Omega|},
\end{equation}
where $B$ denotes the Euclidean ball $$B:=B(x_0,r), \quad r:=\mathrm{dist}(x_0,\partial \Omega).$$ 
\end{cor}

\begin{cor}[Local sufficient condition for a harmonic pseudosphere to be a sphere]\label{c: condizione sufficiente pseudosfere}
Let $\partial \Omega$ be a harmonic pseudosphere centered at $x_0$. If there exists a point $z\in T(\partial \Omega,x_0)$ such that $$S_z(\partial \Omega,x_0)\leq1, $$ then $\partial\Omega$ is a Euclidean sphere centered at $x_0$. Precisely, $$\partial\Omega=\partial B (x_0,|x_0-z|).$$
\end{cor}

\begin{cor}\label{c: pseudosfere lipschitz piatte}
Let $\partial \Omega$ be a harmonic pseudosphere centered at the point $x_0$. If $\partial\Omega$ is Lipschitz flat at a point $z\in T(\partial\Omega,x_0)$, then $\partial \Omega$ is a Euclidean sphere centered at $x_0$. Precisely, $$\partial\Omega=\partial B(x_0,|x_0-z|).$$
\end{cor}

These theorems extend the results of \cite{FriLit}, \cite{Fic86} and, partially, \cite{LewVog02}. Indeed, the result obtained by Lewis and Vogel in \cite{LewVog02} requires very mild assumptions, but on the whole boundary of $\Omega$. Here, instead, we have obtained a harmonic characterization of Euclidean spheres which requires only local properties.

\section{Computation of Some Flatness Indices}\label{s: computations}

In this section we compute the spherical flatness index in some particular cases, proving the statements of Remark \ref{r: indice uguale ad 1} and of Proposition \ref{p: indice sferico nel caso lipschitz piatto}. \\
We start by computing the spherical flatness index in the case in which $\partial\Omega$ is flat in a neighborhood of the point $z$.

\begin{prop}\label{p: indice sferico nel caso piatto}
Let $\Omega\subseteq\mathbb{R}^n$ be a halfspace and let $x_0\in \Omega$. 
 Then  $T(\partial \Omega,x_0)$ is the singleton $\{z\}=\partial \Omega\cap \partial B(x_0,\operatorname{dist}(x_0,\partial \Omega))$ and 
$$S_z(\partial \Omega,x_0)=1.$$
\end{prop}
\begin{proof}
In this proof we consider the splitting of $\mathbb{R}^n$ as $\mathbb{R}^{n-1}\times\mathbb{R}$ and we denote its points as $$(x,y)\in\mathbb{R}^{n-1}_x\times\mathbb{R}_y.$$
Using the invariance with respect to translations and rotations we can assume that $$\Omega=\{(y,t)\in\mathbb{R}^n \ : \ y\in\mathbb{R}^{n-1}, \ t<0\}$$ and we consider the point $x_0=(0,-r)\in \Omega$ with $r>0$. \\ Thus, we have that $T(\partial \Omega,x_0)=\{z\}$ with $z=(0,0)$. \\
We choose a point $$\alpha=(\xi,t)\in\mathbb{R}^{n-1}\times\mathbb{R} \ \mathrm{with} \ t>0.$$ In this way we have that $\alpha\notin\overline{\Omega}$ and we easily notice that $$\left(\alpha\to z \ \mathrm{with} \ \alpha\notin\overline{\Omega}\right) \Leftrightarrow \left(\xi\to0, \ t\to0\right).$$ 
We let $R>0$ and we consider a point $x=(y,0)\in\partial \Omega\cap B(z,R)$. \\
With these choices we have that
\begin{enumerate}
\item[(i)] $\alpha-x=(\xi-y,t)$;
\item[(ii)] $\alpha-x_0=(\xi,t+r)$;
\item[(iii)] $|\alpha-x|^2=|\xi-y|^2+t^2$;
\item[(iv)] $\langle \alpha-x,\alpha-x_0\rangle= (\xi-y)\cdot\xi+t(t+r)$,
\end{enumerate}
where with $\cdot$ we denote the scalar product in $\mathbb{R}^{n-1}$. \\
Hence, we have that $$\frac{2}{n\omega_n r}\frac{\langle\alpha-x,\alpha-x_0\rangle}{|x-\alpha|^n}=\frac{2}{n\omega_n r}\frac{(\xi-y)\cdot\xi + t(t+r)}{(|y-\xi|^2+t^2)^{\frac{n}{2}}}.$$
We define
\begin{equation}\label{e: definizione I1}
I_1:=\frac{t(t+r)}{(|y-\xi|^2+t^2)^{\frac{n}{2}}}
\end{equation}
and
\begin{equation}\label{e: definizione I2}
I_2:=\frac{(y-\xi)\cdot\xi}{(|y-\xi|^2+t^2)^{\frac{n}{2}}}.
\end{equation}
With these definitions at hands we have that 
\begin{equation}\label{e: scomposizione integranda}
\frac{2}{n\omega_n r}\frac{\langle\alpha-x,\alpha-x_0\rangle}{|x-\alpha|^n}=\frac{2}{n\omega_n r}\left(I_1+I_2\right).
\end{equation}
We notice that 
\begin{align*}
\int_{\partial \Omega\cap B(z,R)} I_1 \, d\sigma(x) &= \int_{|y|\leq R} \frac{t(t+r)}{(|\xi-y|^2+t^2)^{\frac{n}{2}}} \, dy \nonumber \\
&= \int_{\mathbb{R}^{n-1}} \frac{t(t+r)}{(|\xi-y|^2+t^2)^{\frac{n}{2}}} \, dy - \int_{|y|\geq R} \frac{t(t+r)}{(|\xi-y|+t^2)^{\frac{n}{2}}} \, dy.
\end{align*}
We start by studying the first term. We notice that it is invariant with respect to translations and so we have
\begin{align*}
\int_{\mathbb{R}^{n-1}} \frac{t(t+r)}{(|\xi-y|^2+t^2)^{\frac{n}{2}}} \, dy &= (t+r) \int_{\mathbb{R}^{n-1}} \frac{t}{(|y|^2+t^2)^{\frac{n}{2}}} \, dy \nonumber \\
&\stackrel{y=t\eta}{=} (t+r) \int_{\mathbb{R}^{n-1}} \frac{1}{(|\eta|^2+1)^{\frac{n}{2}}} \, d\eta.
\end{align*}
Since $\alpha\to z$ if and only if $\xi\to0$ and $t\to0$, we have that 
\begin{equation}\label{e: limite I1 prima parte}
\lim_{t\to0} (t+r)\int_{\mathbb{R}^{n-1}} \frac{1}{(|\eta|^2+1)^{\frac{n}{2}}} \, d\eta = r\frac{n\omega_n}{2}.
\end{equation}
We study now the second term. We notice that 
\begin{align*}
0 &\leq \int_{|y|\geq R} \frac{t(t+r)}{(|y-\xi|^2+t^2)^{\frac{n}{2}}} \, dy \nonumber \\
&\leq (t+r)\int_{|y|\geq R} \frac{t}{|y-\xi|^n} \, dy \nonumber \\
&\stackrel{\eta=y-\xi}{=} t(t+r)\int_{|\eta|\geq R-1} \frac{1}{|\eta|^n} \, d\eta,
\end{align*}
where in the last step we have assumed that $|\xi|<1$, since we want to compute the limit as $\alpha$ goes to $z$. \\
Now, we use the polar coordinates $\eta=\rho\theta$ with $\rho>0$ and $\theta\in\mathbb{R}^{n-1}$ with $|\theta|=1$. In this way we get
\begin{align*}
t(t+r)\int_{|\eta|\geq R-1} \frac{1}{|\eta|^n} \, d\eta &= t(t+r)\sigma_{n-1} \int_{R-1}^{+\infty} \frac{1}{\rho^2} \, d\rho \nonumber \\
&= t(t+r)\sigma_{n-1} \frac{1}{R-1}
\end{align*}
and this clearly converges to $0$ as $t\to0$. \\
Hence, we have that
\begin{equation}\label{e: limite I1 seconda parte}
\lim_{t\to0} \int_{|y|\geq R} \frac{t(t+r)}{(|\xi-y|^2+t^2)^{\frac{n}{2}}} \, dy =0.
\end{equation}
Recalling the definition of $I_1$ in \eqref{e: definizione I1} and combining \eqref{e: limite I1 prima parte} and \eqref{e: limite I1 seconda parte} we get
\begin{equation}\label{e: limite I1}
\lim_{\pedfrac{\alpha\to z}{\alpha\notin\overline{\Omega}}} \frac{2}{n\omega_n r} \int_{\partial \Omega\cap B(z,R)} I_1 \, d\sigma(x) =1.
\end{equation}
Now, we have to study $$\frac{2}{n\omega_n r} \int_{\partial \Omega\cap B(z,R)} I_2 \, d\sigma(x).$$ Recalling the definition of $I_2$ in \eqref{e: definizione I2}, we have that this term is given, up to the constant, by the following integral
\begin{equation}\label{e: definizione I(xi,t)}
I(\xi,t):=\int_{|y|\leq R} \frac{(\xi-y)\cdot\xi}{(|\xi-y|^2+t^2)^{\frac{n}{2}}} \, dy.
\end{equation}
Since, for $R>0$ fixed, we want to study the behavior of this integral as $(\xi,t)\to(0,0)$ we can assume that $\displaystyle |\xi|<\frac{R}{2}$. \\
We notice that
$$ I(\xi,t) = \int_{B(0,R)\cap B(\xi,\frac{R}{4})} \frac{(\xi-y)\cdot\xi}{(|\xi-t|^2+t^2)^{\frac{n}{2}}} \, dy + \int_{B(0,R)\backslash B(\xi,\frac{R}{4})} \frac{(\xi-y)\cdot\xi}{(|\xi-t|^2+t^2)^{\frac{n}{2}}} \, dy.$$
We define
\begin{equation}\label{e: definizione I1(xi,t)}
I^{(1)}(\xi,t):= \int_{B(0,R) \cap B(\xi,\frac{R}{4})} \frac{(\xi-y)\cdot\xi}{(|\xi-t|^2+t^2)^{\frac{n}{2}}} \, dy
\end{equation}
and
\begin{equation}\label{e: definizione I2(xi,t)}
I^{(2)}(\xi,t):= \int_{B(0,R) \backslash B(\xi,\frac{R}{4})} \frac{(\xi-y)\cdot\xi}{(|\xi-t|^2+t^2)^{\frac{n}{2}}} \, dy
\end{equation}
getting
\begin{equation}\label{e: scomposizione I(xi,t)}
I(\xi,t)=I^{(1)}(\xi,t)+I^{(2)}(\xi,t).
\end{equation}
We start by studying $I^{(1)}(\xi,t)$. \\
Since we have assumed that $\displaystyle |\xi|<\frac{R}{2}$, then $B(\xi,\frac{R}{4})\subseteq B(0,R)$. Hence, we have
\begin{align*}
I^{(1)}(\xi,t)&=\int_{B(\xi,\frac{R}{4})} \frac{(\xi-y)\cdot\xi}{(|\xi-y|^2+t^2)^{\frac{n}{2}}} \, dy \nonumber \\
&\stackrel{\eta=\xi-y}{=} \int_{B(0,\frac{R}{4})} \frac{\eta\cdot\xi}{(|\eta|^2+t^2)^{\frac{n}{2}}} \, d\eta \nonumber \\
&= \int_0^{\frac{R}{4}} \frac{1}{(\rho^2+t^2)^{\frac{n}{2}}} \biggl(\int_{|\eta|=\rho} \eta \cdot \xi \, d\sigma(\eta) \biggr) d\eta,
\end{align*}
where in the last step we have applied the coarea formula. \\
We notice that if $T$ is a rotation around the origin, then 
\begin{align*}
\int_{|\eta|=\rho} \eta \cdot \xi \, d\sigma(\eta) &= \int_{|\eta|=\rho} T\eta \cdot\xi \, d\sigma(\eta) \nonumber \\
&= \int_{|\eta|=1} \eta \cdot T\xi \, d\sigma(\eta).
\end{align*}
We choose a rotation $T$ such that $$T\xi=|\xi|e_1,$$ where $e_1=(1,0,\ldots,0)$. In this way we have that 
\begin{equation}\label{e: integrale I1 con rotazione}
\int_{|\eta|=1} \eta\cdot T\xi \, d\sigma(\eta)=|\xi|\int_{|\eta|=1} \eta_1 \, d\sigma(\eta),
\end{equation}
where $\eta_1$ is the first coordinate of $\eta$. \\
Now, we can rewrite \eqref{e: integrale I1 con rotazione} as 
\begin{align*}
|\xi|\int_{|\eta|=1} \eta_1 \, d\sigma(\eta) &= |\xi|\left( \int_{\pedfrac{|\eta|=1}{\eta_1\geq0}} \eta_1 \, d\sigma(\eta) + \int_{\pedfrac{|\eta|=1}{\eta_1\leq0}} \eta_1 \, d\sigma(\eta) \right) \nonumber \\
&= |\xi| \left( \int_{\pedfrac{|\eta|=1}{\eta_1\geq0}} \eta_1 \, d\sigma(\eta) - \int_{\pedfrac{|\eta|=1}{\eta_1\geq0}} \eta_1 \, d\sigma(\eta) \right) \nonumber \\
&=0.
\end{align*}
Thus, we have proved that
\begin{equation}\label{e: I1 uguale 0}
I^{(1)}(\xi,t)=0 \quad \forall\xi \ : \ |\xi|\leq\frac{R}{4}.
\end{equation}
Now, we want to study the behavior of $I^{(2)}(\xi,t)$. \\
Recalling the definition of $I^{(2)}(\xi,t)$ in \eqref{e: definizione I2(xi,t)} and applying the Cauchy - Schwartz inequality we have
\begin{align*}
|I^{(2)}(\xi,t)| &\leq \int_{B(0,R)\backslash B(\xi,\frac{R}{4})} \frac{|\xi-y| |\xi|}{(|\xi-y|^2+t^2)^{\frac{n}{2}}} \, dy \nonumber \\
&\leq 2R|\xi| \int_{B(0,R)} \frac{1}{\left(\left(\frac{R}{4}\right)^2+t^2\right)^{\frac{n}{2}}} \, dy \nonumber \\
&= 2R|\xi| \left(\frac{4}{R}\right)^{\frac{n}{2}}.
\end{align*}
Hence, for every fixed $R>0$, we have that
\begin{equation}\label{e: limite I2 uguale 0}
\lim_{(\xi,t)\to(0,0)} I^{(2)}(\xi,t)=0.
\end{equation}
Thus, recalling \eqref{e: scomposizione I(xi,t)} and combining \eqref{e: I1 uguale 0} and \eqref{e: limite I2 uguale 0}, we have that for every $R>0$
\begin{equation}\label{e: limite I(xi,t) uguale 0}
\lim_{(\xi,t)\to(0,0)} I(\xi,t)=0.
\end{equation}
Hence, recalling \eqref{e: scomposizione integranda} and combining \eqref{e: limite I1} and \eqref{e: limite I(xi,t) uguale 0}, we have proved that $$S_z(\partial \Omega,x_0)=1.$$
\end{proof}

We consider now the case in which $\partial\Omega$ is a Euclidean sphere containing the point $x_0$.
\begin{prop}\label{p: indice sferico nel caso sfera non centrata}
Let $x_0\in\mathbb{R}^n$ and $R>0$. Assume that $\Omega$ is the Euclidean ball $B(x_0,R)$ and let $x_1\in\Omega$. \\ Then, if $x_1\neq x_0$, $T(\partial \Omega,x_1)$ is the singleton $\{z\}=\partial \Omega\cap \partial B(x_1,\operatorname{dist}(x_1,\partial \Omega))$ and $$S_z(\partial\Omega,x_1)=1.$$ Instead, if $x_1=x_0$, then $T(\partial\Omega,x_1)=\partial\Omega$ and $$S_z(\partial\Omega,x_1)=1 \quad \forall z\in\partial\Omega.$$
\end{prop}

To prove this Proposition, we start by proving the following lemma in which we assume that $\Omega$ is a Euclidean ball centered at the point $x_0$.
\begin{lemma}\label{l: indice sferico nel caso sfera centrata}
Let $x_0\in\mathbb{R}^n$ and $r>0$. Assume that $\Omega$  is the Euclidean ball $B(x_0,r)$. Then, we have that $T(\partial \Omega,x_0)=\partial\Omega$ and
$$S_z(\partial\Omega,x_0)=1 \quad \forall z\in\partial\Omega.$$
\end{lemma}
\begin{proof}
First of all, we remark that
\begin{align*}
\frac{|x-x_0|^2-|\alpha-x_0|^2}{|x-\alpha|^n} &= \frac{|(x-\alpha)+(\alpha-x_0)|^2-|\alpha-x_0|^2}{|x-\alpha|^n} \nonumber \\
&= \frac{|x-\alpha|^2 + 2\langle x-\alpha, \alpha-x_0 \rangle}{|x-\alpha|^n} \nonumber \\
&= \frac{1}{|x-\alpha|^{n-2}}-2\frac{\langle \alpha-x, \alpha-x_0 \rangle}{|x-\alpha|^n}.
\end{align*}
From this it follows that
\begin{equation} \label{e: scomposizione integranda sferica}
2\frac{\langle \alpha-x, \alpha-x_0 \rangle}{|x-\alpha|^n} = \frac{1}{|x-\alpha|^{n-2}} + \frac{|\alpha-x_0|^2-|x-x_0|^2}{|x-\alpha|^n}.
\end{equation}
We define 
\begin{equation} \label{e: definizione k_alpha}
k_{\alpha}(x):=\frac{|\alpha-x_0|^2-|x-x_0|^2}{|x-\alpha|^n}
\end{equation}
and so by \eqref{e: scomposizione integranda sferica} it follows that for every $R>0$
\begin{equation} \label{e: scomposizione integrale sfera}
2\int_{\partial \Omega \cap B(z,R)} \frac{\langle \alpha-x, \alpha-x_0 \rangle}{|x-\alpha|^n} \, d\sigma(x) = \int_{\partial \Omega \cap B(z,R)} \left( \frac{1}{|x-\alpha|^{n-2}} + k_{\alpha}(x) \right) d\sigma(x).
\end{equation}
We start by studying the first term. Applying Lebesgue dominated convergence theorem we have that
\begin{align*} 
\lim_{\pedfrac{\alpha\to z}{\alpha\notin\overline{\Omega}}} \int_{\partial \Omega \cap B(z,R)} \frac{1}{|x-\alpha|^{n-2}} \, d\sigma(x) &= \int_{\partial \Omega \cap B(z,R)} \frac{1}{|x-z|^{n-2}} \, d\sigma(x) \nonumber \\
&\stackrel{y=x-z}{=}\int_{(-z+\partial \Omega)\cap B(0,R)} \frac{1}{|y|^{n-2}} \, d\sigma(y).
\end{align*}
In the last integral, we use the polar coordinates in $\mathbb{R}^{n-1}$, letting $y=\rho\omega$ with $\rho\in(0,R)$ and $\omega\in\mathbb{R}^{n-1}$ with $|\omega|=1$. In this way we get
$$ \int_{(-z+\partial \Omega)\cap B(0,R)} \frac{1}{|y|^{n-2}} \, d\sigma(y) = \sigma_{n-1} \int_{0}^R \frac{\rho^{n-2}}{\rho^{n-2}} \, d\rho = \sigma_{n-1} R$$
and so
\begin{equation}\label{e: limite primo integrale sferico}
\lim_{R\to0}\left(\lim_{\pedfrac{\alpha\to z}{\alpha\notin\overline{\Omega}}} \int_{\partial \Omega \cap B(z,R)} \frac{1}{|x-\alpha|^{n-2}} \, d\sigma(x) \right)=0.
\end{equation}
Now, we study the second term. First of all, we rewrite it as
\begin{equation}\label{e: scomposizione integrale k_alpha}
\int_{\partial \Omega \cap B(z,R)} k_{\alpha}(x) \, d\sigma(x) = \int_{\partial \Omega} k_{\alpha}(x) \, d\sigma(x) - \int_{\partial \Omega \backslash B(z,R)} k_{\alpha} (x) \, d\sigma(x).
\end{equation}
We notice that $k_{\alpha}$ is harmonic in $\mathbb{R}^n\backslash\{\alpha\}$ if $\alpha\notin\overline{\Omega}$ since it is obtained as difference of two harmonic functions. Then, since $\partial \Omega = \partial B(x_0,r)$ is a sphere, we can apply Gauss mean value theorem getting 
\begin{align*}
\int_{\partial \Omega} k_{\alpha} \, d\sigma(x) &= n\omega_n r^{n-1} k_{\alpha}(x_0) \nonumber \\
&= n\omega_n r^{n-1} \frac{|\alpha-x_0|^2-0}{|x_0-\alpha|^n}.
\end{align*}
Thus, we have proved that
\begin{equation}\label{e: integrale 1 k_alpha}
\int_{\partial \Omega} k_{\alpha}(x) \, d\sigma(x) = n\omega_n r.
\end{equation}
Applying Lebesgue dominated convergence theorem we have 
\begin{equation}\label{e: integrale 2 k_alpha}
\lim_{\pedfrac{\alpha\to z}{\alpha\notin\overline{\Omega}}}\int_{\partial \Omega \backslash B(z,R)} k_{\alpha}(x) d\sigma(x) = \int_{\partial \Omega \backslash B(z,R)} \frac{r^2-|x-x_0|^2}{|x-z|^n} \, d\sigma(x)=0.
\end{equation}
Thus, recalling \eqref{e: scomposizione integrale k_alpha} and combining \eqref{e: integrale 1 k_alpha} and \eqref{e: integrale 2 k_alpha}, we have 
\begin{equation}\label{e: limite secondo integrale sferico}
\lim_{\pedfrac{\alpha\to z}{\alpha\notin\overline{\Omega}}} \int_{\partial \Omega\cap B(z,R)} k_{\alpha}(x) \, d\sigma(x)=n\omega_n r.
\end{equation}
Now, using \eqref{e: limite primo integrale sferico} and \eqref{e: limite secondo integrale sferico} in \eqref{e: scomposizione integrale sfera}, we get
$$\lim_{R\to0}\left(\lim_{\pedfrac{\alpha\to z}{\alpha\notin\overline{\Omega}}} 2\int_{\partial \Omega \cap B(z,R)} \frac{\langle \alpha-x, \alpha-x_0 \rangle}{|x-\alpha|^n} \, d\sigma(x)\right) = n\omega_n r.$$
Hence, recalling the definition of spherical flatness index we have proved that $$S_z(\partial \Omega,x_0)=1.$$
\end{proof}

With this lemma at hands, we are now ready to prove Proposition \ref{p: indice sferico nel caso sfera non centrata}.
\begin{proof}[Proof of Proposition \ref{p: indice sferico nel caso sfera non centrata}]
As we have done in the previous proof, we consider the splitting of $\mathbb{R}^n$ as $\mathbb{R}^{n-1}\times\mathbb{R}$ and we denote its points as $$(x,y)\in\mathbb{R}^{n-1}_x\times\mathbb{R}_y.$$
Since by Remark \ref{r: invarianza indice sferico} the spherical flatness index is invariant with respect under rigid motions, we can assume that $R=1$ and $x_0=(0,-1)=-e_n$, where $e_n=(0,\ldots,0,1)$. \\
We let $0<r<R$ and we consider the point $x_1=(0,-r)\in\Omega$ and the point $z=(0,0)\in T(\partial\Omega,x_1)$. \\
We notice that with these choices we have that $|z-x_1|=r$ and so by definition of spherical flatness index we have that 
\begin{equation}\label{e: indice sferico sfera non centrata}
S_z(\partial\Omega,x_1)=\frac{2}{n\omega_n r} \liminf_{\rho\searrow0}\left(\liminf_{\pedfrac{\alpha\to z}{\alpha\notin\overline{\Omega}}} \int_{\partial\Omega\cap B(z,\rho)} \frac{\langle \alpha-x, \alpha-x_1\rangle}{|x-\alpha|^n} \, d\sigma(x) \right).
\end{equation}
We notice that  $$\langle \alpha-x, \alpha-x_1\rangle =  \langle \alpha-x, \alpha-x_0\rangle  + \langle \alpha-x, x_0-x_1\rangle.$$
Thus, letting
\begin{equation}\label{e: definizione i_alpha}
I_{\alpha}:=\int_{\partial\Omega \cap B(z,\rho)} \frac{\langle \alpha-x, \alpha-x_0 \rangle}{|x-\alpha|^n} \, d\sigma(x)
\end{equation}
and
\begin{equation}\label{e: definizione j_alpha}
J_{\alpha}:=\int_{\partial\Omega\cap B(z,\rho)} \frac{\langle \alpha-x, x_0-x_1 \rangle}{|x-\alpha|^n} \, d\sigma(x),
\end{equation}
we have that
\begin{equation}\label{e: splitting integrale sfera non centrata}
\int_{\partial\Omega\cap B(z,\rho)} \frac{\langle \alpha-x, \alpha-x_1 \rangle}{|x-\alpha|^n} \, d\sigma(x) =I_{\alpha}+J_{\alpha}.
\end{equation}
Hence, combining \eqref{e: indice sferico sfera non centrata} and \eqref{e: splitting integrale sfera non centrata}, we have that
\begin{equation}\label{e: indice sferico sfera non centrata con splitting}
S_z(\partial\Omega,x_1)=\frac{2}{n\omega_n r}\liminf_{\rho\searrow0}\left(\liminf_{\pedfrac{\alpha\to z}{\alpha\notin\overline{\Omega}}}\left(I_{\alpha}+J_{\alpha}\right)\right).
\end{equation}
We start by studying the limit of $I_{\alpha}$. \\
We notice that 
\begin{align*}
\frac{2}{n\omega_n r}\lim_{\alpha\to z}I_{\alpha} &= \frac{R}{r} \lim_{\alpha\to z} \frac{2}{n\omega_nR}I_{\alpha} \nonumber \\
&=\frac{R}{r}(1+o(\rho)),
\end{align*} where in the last equality we have applied Lemma \ref{l: indice sferico nel caso sfera centrata} since $\Omega$ is a Euclidean ball centered at the point $x_1$ with radius $R$. \\
Hence, we have proved that
\begin{equation}\label{e: limite i_alpha}
\frac{2}{n\omega_n r}\lim_{\rho\to0}\left(\lim_{\alpha\to z} I_{\alpha}\right)=\frac{R}{r}\stackrel{R=1}{=}\frac{1}{r}.
\end{equation}
Now, we want to study the limit of $J_{\alpha}$. \\
First of all, we rewrite it as 
\begin{align*}
J_{\alpha} &= \int_{\partial\Omega \cap B(z,\rho)} \frac{\langle \alpha-x, x_0-x_1 \rangle}{|x-\alpha|^n} \, d\sigma(x) \nonumber \\
&=\int_{\partial\Omega} \frac{\langle \alpha-x, x_0-x_1 \rangle}{|x-\alpha|^n} \, d\sigma(x) - \int_{\partial\Omega\backslash B(z,\rho)} \frac{\langle \alpha-x, x_0-x_1 \rangle}{|x-\alpha|^n} \, d\sigma(x).
\end{align*}
Since the map $$x\mapsto\frac{\langle \alpha-x, x_0-x_1 \rangle}{|x-\alpha|^n}$$ is harmonic in $\mathbb{R}^n\backslash\{\alpha\}$ and $\partial\Omega$ is a Euclidean sphere centered at the point $x_0$, we can apply Gauss mean value theorem for harmonic functions getting
\begin{align*}
\int_{\partial\Omega}\frac{\langle \alpha-x, x_0-x_1 \rangle}{|x-\alpha|^n} \, d\sigma(x) &= n\omega_n R^{n-1} \fint_{\partial\Omega} \frac{\langle \alpha-x, x_0-x_1 \rangle}{|x-\alpha|^n} \, d\sigma(x) \nonumber \\
&=n\omega_n R^{n-1}\frac{\langle \alpha-x_0, x_0-x_1 \rangle}{|x_0-\alpha|^n}.
\end{align*}
Thus, we have that 
\begin{equation}\label{e: limite 1 primo termine j_alpha}
\lim_{\alpha\to z} \int_{\Omega} \frac{\langle \alpha-x, x_0-x_1 \rangle}{|x-\alpha|^n} \, d\sigma(x) = n\omega_n R^{n-1} \frac{\langle z-x_0, x_0-x_1 \rangle}{|x_0-z|^n}.
\end{equation}
With our choices, we have
\begin{itemize}
\item[(i)] $z-x_0=(0,R)$;
\item[(ii)] $x_0-x_1=(0,-R+r)$; 
\item[(iii)] $\langle z-x_0, x_0-x_1 \rangle = (-R+r)R$;
\item[(iv)] $|x_0-z|^n=R^n$.
\end{itemize}
Hence, from \eqref{e: limite 1 primo termine j_alpha} it follows that 
\begin{equation}\label{e: limite primo termine j_alpha} 
\lim_{\alpha\to z} \int_{\Omega} \frac{\langle \alpha-x, x_0-x_1 \rangle}{|x-\alpha|^n} \, d\sigma(x)=n\omega_nR^{n-1}\frac{R(-R+r)}{R^n}=n\omega_n(-R+r).
\end{equation}
Now we have to study $\displaystyle \int_{\partial\Omega\backslash B(z,\rho)} \frac{\langle \alpha-x, x_0-x_1\rangle}{|x-\alpha|^n} \, d\sigma(x)$. \\
Applying Lebesgue dominated convergence theorem, we have that 
\begin{equation}\label{e: limite 1 secondo termine j_alpha}
\lim_{\alpha\to z} \int_{\partial\Omega \backslash B(z,\rho)} \frac{\langle \alpha-x, x_0-x_1 \rangle}{|x-\alpha|^n} \, d\sigma(x) = \int_{\partial\Omega \backslash B(z,\rho)} \frac{\langle z-x, x_0-x_1 \rangle}{|z-x|^n} \, d\sigma(x).
\end{equation}
With our choices we have that $$\langle z-x, x_0-x_1 \rangle = -x_n(-R+r)$$ and so from \eqref{e: limite 1 secondo termine j_alpha} it follows that
$$\lim_{\alpha\to z} \int_{\partial\Omega \backslash B(z,\rho)} \frac{\langle \alpha-x, x_0-x_1 \rangle}{|x-\alpha|^n} \, d\sigma(x)=(R-r)\int_{\partial\Omega \backslash B(z,\rho)} \frac{x_n}{|x|^n} \, d\sigma(x).$$
Applying Beppo Levi theorem we have $$\lim_{\rho\to0} (R-r)\int_{\partial\Omega \backslash B(z,\rho)} \frac{x_n}{|x|^n} \, d\sigma(x) =(R-r)\int_{\partial\Omega} \frac{x_n}{|x|^n} \, d\sigma(x).$$
Using the following fact (see Appendix)
\begin{equation}\label{e: integrale notevole}
\int_{\partial\Omega}\frac{x_n}{|x|^n} \, d\sigma(x)=-\frac{n\omega_n}{2},
\end{equation}
we have that
\begin{equation}\label{e: limite secondo termine j_alpha}
\lim_{\rho\to0}\left(\lim_{\alpha\to z} \int_{\partial\Omega \backslash B(z,\rho)} \frac{\langle \alpha-x, x_0-x_1 \rangle}{|x-\alpha|^n} \, d\sigma(x)\right) = -\frac{n\omega_n}{2}(R-r).
\end{equation}
Thus, combining \eqref{e: limite primo termine j_alpha} and \eqref{e: limite secondo termine j_alpha}, and recalling that we have chosen $R=1$ we get
\begin{equation}\label{e: limite j_alpha}
\lim_{\rho\to0}\left(\lim_{\alpha\to z} \frac{2}{n\omega_n r} J_{\alpha}\right)= 1-\frac{1}{r}.
\end{equation}
Thus, combining \eqref{e: splitting integrale sfera non centrata}, \eqref{e: limite i_alpha} and \eqref{e: limite j_alpha} we have that $$S_z(\partial\Omega,x_1)=1.$$
\end{proof}

Finally, we prove Proposition \ref{p: indice sferico nel caso lipschitz piatto}.
\begin{proof}[Proof of Proposition \ref{p: indice sferico nel caso lipschitz piatto}]
Using the invariance of the spherical index w.r.t. translations, we may assume without lost of generality that $x_0=0$. \\
We divide the proof of this proposition into steps. 

\medbreak

\textsc{\textbf{Step 1.}} In this step we try to characterize the boundary of $\Omega$ near a point $z\in T(\partial \Omega,0)$. \\ Using the invariance w.r.t. rotations, we can assume that $z=(0,r)\in\mathbb{R}^{n-1}\times\mathbb{R}$, where $r=\mathrm{dist}(0,\partial \Omega)$. \\
Since $\partial\Omega$ is Lipschitz flat at the point $z\in T(\partial\Omega,0)$, there exist positive constants $a,b>0$ and $R_0>0$, with $a<r<b$, and a function $\psi:\widehat{B}_{R_0}(0)\rightarrow(a,b)$ such that
\begin{itemize}
      \item[(i)] $\psi$ is Lipschitz, differentiable and $\nabla\psi(t\xi)\to\nabla\psi(0)$ as $t\to0$ 
			             for a.e. $\xi\in\mathbb{R}^{n-1}$;
			\item[(ii)] $\partial \Omega \cap U_{R_0}(a,b)=\left\{(y,\psi(y)) \ : \ y\in\widehat{B}_{R_0}(0)\right\}$;
			\item[(iii)] $ \Omega \cap U_{R_0}(a,b)=\left\{(y,t) \ : \ y\in\widehat{B}_{R_0}(0), \ a<t<\psi(y)<b\right\}$;
			\item[(iv)] $z=(0,r)=(0,\psi(0))$;
			\item[(v)] $\nabla\psi(0)=0$ since the ball $B(0,r)$ is tangent to $\partial \Omega$ at $z$.
\end{itemize}

We take $0<R<R_0$ such that $B(z,R)\subseteq U_{R_0}(a,b)$. In this way we have that the generic point of $\partial \Omega\cap B(z,R)$ will be $$x=(y,\psi(y)),$$ with $y\in\mathbb{R}^{n-1}$ and $|y|< R$. \\
Since $z$ is orthogonally accessible, we choose a point $$\alpha=(0,t)\in\mathbb{R}^{n-1}\times\mathbb{R} \ \mathrm{with} \ r<t<b.$$ In this way we have have that $\alpha\notin\overline{\Omega}$ and we easily notice that $$ (\alpha\downarrow z \ \mathrm{with} \ \alpha\notin\overline{\Omega}) \ \Leftrightarrow \ t\searrow r,$$ where with the notation $\alpha\downarrow z$ we mean that  
$\alpha$ converges orthogonally to $z$.

\medbreak

\textsc{\textbf{Step 2.}} In this step we construct the term $I_R(\alpha)$.  \\ 
The term of which we want to compute the limit in the definition of spherical index is the following
\begin{equation}\label{e: I_R(alpha) LP parte 1}
I_R(\alpha)= \frac{2}{n\omega_n r}\int_{\partial\Omega\cap B(z,R)} \frac{\langle\alpha-x,\alpha\rangle}{|x-\alpha|^n} \, d\sigma(x).
\end{equation}
We know that if $x\in\partial \Omega\cap B(z,R)$ then $x=(y,\psi(y))$ with $y\in\mathbb{R}^{n-1}$ and $|y|< R$. Since we have chosen the point $\alpha=(0,t)$, then we have that
\begin{equation}\label{e: differenza}
\alpha-x=(-y,t-\psi(y)) \ \mathrm{with} \ |y|< R.
\end{equation}
Thus, recalling that  $\partial\Omega\cap B(z,R)$ is the graph of a Lipschitz function and keeping in mind \eqref{e: differenza}, we have that \eqref{e: I_R(alpha) LP parte 1} becomes  
\begin{equation}\label{e: I_R(alpha) caso LP} 
I_R(\alpha)=\frac{2t}{n\omega_n r}\int_{\left\{|y|<R\right\}}\frac{t-\psi(y)}{\left(|y|^2+(t-\psi(y))^2\right)^{\frac{n}{2}}}N(y) \, dy
\end{equation} 
where $$N(y)=\sqrt{1+|\nabla\psi(y)|^2}.$$
Using the coarea formula the last integral of \eqref{e: I_R(alpha) caso LP} is equal to $$\int_0^{R}\left(\int_{\left\{|y|=\rho\right\}}\frac{t-\psi(y)}{\left(\rho^2+(t-\psi(y))^2\right)^{\frac{n}{2}}}N(y) \, d\sigma(y)\right)d\rho.$$
We consider the change of variables $y=\rho\eta$ with $\eta\in\mathbb{R}^{n-1}$, $|\eta|=1$. In this way we get 
$$\int_0^R\left(\int_{\left\{|\eta|=1\right\}}\rho^{n-2}\frac{t-\psi(\rho\eta)}{\left(\rho^2+(t-\psi(\rho\eta))^2\right)^{\frac{n}{2}}}N(\rho\eta) \, d\sigma(\eta)\right)d\rho.$$ 
Changing the order of integration and making the change of variables $\rho=(t-r)s$ we get
\begin{equation}\label{e: I_R(alpha) caso LP CV}
I_R(\alpha)=\frac{2t}{n\omega_n r}\int_{\left\{|\eta|=1\right\}}\biggl(\int_0^{\frac{R}{t-r}}\frac{s^{n-2}}{\left(s^2+\left(\frac{t-\psi}{t-r}\right)^2\right)^{\frac{n}{2}}}\frac{t-\psi}{t-r}N \, ds\biggr)d\sigma(\eta)
\end{equation}
where $$N:=N((t-r)s\eta) \quad \mathrm{and} \quad \psi:=\psi((t-r)s\eta).$$

\medbreak

\textsc{\textbf{Step 3.}} In this step we want to study the behavior of $I_R(\alpha)$ when $\alpha\downarrow z$ (i.e. $t\to r^{+}$). In particular, we want to analyze the behavior as $t\to r^{+}$ of $$\frac{t-\psi}{t-r}$$ where $\psi=\psi((t-r)s\eta)$. \\
In order to do this we define
\begin{equation}\label{e: definizione omega} 
\omega(y):=\frac{\psi(y)-\psi(0)}{|y|} \quad \mathrm{with} \ y\neq0.
\end{equation}
Since $\psi$ is differentiable and $\nabla\psi(0)=0$, then we get 
\begin{equation}\label{e: limite omega}
\omega(y)\to\langle\nabla\psi(0),y\rangle=0 \ \mathrm{as} \ y\to0.
\end{equation} 
Since $\psi$ is Lipschitz, we have that $$|\omega(y)|=\frac{|\psi(y)-\psi(0)|}{|y|}\leq L,$$ where $L$ denotes the Lipschitz constant of $\psi$. Hence, $\omega$ is bounded. \\
By the definition of the function $\omega$, it follows that $$\psi(y)-r=\psi(y)-\psi(0)=|y|\omega(y)$$ and so, recalling that $y=(t-r)s\eta$ with $|\eta|=1$, we get $$\frac{t-\psi(y)}{t-r}=1-\frac{|y|\omega(y)}{t-r}=1-s\omega((t-r)s\eta).$$
Thus, we can rewrite \eqref{e: I_R(alpha) caso LP CV} as 
\begin{equation}\label{e: I_R(alpha) omega} 
I_R(\alpha)=\frac{2t}{n\omega_n r}\int_{\left\{|\eta|=1\right\}}\biggl(\int_0^{\frac{R}{t-r}}\frac{s^{n-2}}{(s^2+(1-s\omega)^2)^{\frac{n}{2}}}(1-s\omega)N \, ds\biggr)d\sigma(\eta),
\end{equation}
where we use the notation $$\omega:=\omega((t-r)s\eta).$$
We define
\begin{equation}\label{e: I_R(alpha) omega 1}
I_R^{(1)}(\alpha)=\frac{2t}{n\omega_n r}\int_{\left\{|\eta|=1\right\}}\biggl(\int_0^{\frac{R}{t-r}}\frac{s^{n-2}}{(s^2+(1-s\omega)^2)^{\frac{n}{2}}}N \, ds\biggr)d\sigma(\eta)
\end{equation}
and
\begin{equation}\label{e: I_R(alpha) omega 2}
I_R^{(2)}(\alpha)=\frac{2t}{n\omega_n r}\int_{\left\{|\eta|=1\right\}}\biggl(\int_0^{\frac{R}{t-r}}\frac{s^{n-1}}{(s^2+(1-s\omega)^2)^{\frac{n}{2}}}\omega N \, ds\biggr)d\sigma(\eta).
\end{equation}
With these definitions at hands we rewrite \eqref{e: I_R(alpha) omega} as
\begin{equation}\label{e: splitting I_R(alpha)}
I_R(\alpha)=I_R^{(1)}(\alpha)-I_R^{(2)}(\alpha).
\end{equation}

\medbreak 

\textsc{\textbf{Step 4.}} In this step we want to study the behavior of $I_R^{(1)}(\alpha)$. In particular, we want to prove that 
\begin{equation}\label{e: limite I_R(alpha)^1}
\lim_{\pedfrac{\alpha\downarrow z}{\alpha\notin\overline{\Omega}}}I_R^{(1)}(\alpha)=1.
\end{equation} 
First of all, we recall that $N=N(s(t-r)\eta)=\sqrt{1+|\nabla\psi(s(t-r)\eta)|^2}$ and so we have 
\begin{equation}\label{e: gradiente limitato}
\sup_{s\in\left[0,\frac{R}{t-r}\right]}N(s(t-r)\eta)\leq\sqrt{1+||\nabla\psi||^2_{L^{\infty}(\widehat{B}_R(0))}}=:C<+\infty,
\end{equation}
since the gradient of a Lipschitz and differentiable function is bounded. \\
By using the fact that $\nabla\psi(t\xi)\to\nabla\psi(0)$ as $t\to0$ for a.e. $\xi\in\mathbb{R}^{n-1}$ and $\nabla\psi(0)=0$, we have that
\begin{equation}\label{e: limite N}
N=N(s(t-r)\eta)=1+o(1) \ \mathrm{as} \ t\to r^+
\end{equation}
for a.e. $\eta\in\mathbb{R}^{n-1}$, $|\eta|=1$ and $s\in\left[0,a\right]$ with $a>0$ fixed. \\
Now, in order to compute the limit of \eqref{e: limite I_R(alpha)^1}, we want to show that it is possible to take the limit inside the integral. \\
We fix $a>1$ such that $0<\frac{1}{aL}<\frac{R}{t-r}$. Then, since $\omega$ is bounded by $L$, we have that for all $s\in\left[0,\frac{1}{aL}\right]$ there holds $$|s\omega|\leq\frac{1}{a} \quad \forall s\in\left[0,\frac{1}{aL}\right].$$
Using the superadditivity of the function $x\mapsto x^{\frac{n}{2}}$ and the previous estimates, we have
$$ \frac{s^{n-2}}{\left(s^2+\left(1-s\omega\right)^2\right)^{\frac{n}{2}}} N\leq
\left\{
     \begin{array}{ll} 
		      \frac{s^{n-2}}{(1-s\omega)^n}N\leq C\left(\frac{a}{a-1}\right)^n s^{n-2} & \mbox{if } 
					s\in\left[0,  \frac{1}{aL} \right] \\
					\frac{s^{n-2}}{s^n}N\leq \frac{C}{s^2} & \mbox{if } 
					s\in\left(\frac{1}{aL},\frac{R}{t-r}\right)
		 \end{array}
.
\right. $$
Thanks to this estimate, we can compute the limit in \eqref{e: limite I_R(alpha)^1} using Lebesgue dominated convergence theorem. \\ We recall that as $t\to r^+$ we have 
\begin{itemize}
\item[(i)] $\omega\to0$ by \eqref{e: limite omega}; 
\item[(ii)] $N\to1$ a.e. by \eqref{e: limite N}.
\end{itemize} 
Therefore, we get
$$\lim_{\pedfrac{\alpha\downarrow z}{\alpha\notin\overline{\Omega}}} I_R^{(1)}(\alpha)=\frac{2}{n\omega_n}\sigma_{n-1}\int_0^{+\infty}\frac{s^{n-2}}{(1+s^2)^{\frac{n}{2}}} \, ds= 1$$ and this proves \eqref{e: limite I_R(alpha)^1}.

\medbreak

\textsc{\textbf{Step 5.}} In this step we study the behavior of $I_R^{(2)}(\alpha)$. \\
We let $$\omega=\omega^{+}-\omega^{-}=\max\left\{0,\omega\right\}-\max\left\{0,-\omega\right\}$$ and, recalling the definition of $I_R^{(2)}(\alpha)$ in \eqref{e: I_R(alpha) omega 2}, we define 
\begin{equation}\label{e: ialfa2+}
\left(I_R^{(2)}(\alpha)\right)^{+}=\frac{2t}{n\omega_n r}\int_{\left\{|\eta|=1\right\}}\biggl(\int_0^{\frac{R}{t-r}}\frac{s^{n-1}}{(s^2+(1-s\omega)^2)^{\frac{n}{2}}}\omega^{+} N \, ds\biggr)d\sigma(\eta)
\end{equation} and 
\begin{equation}\label{e: ialfa2-}
\left(I_R^{(2)}(\alpha)\right)^{-}=\frac{2t}{n\omega_n r}\int_{\left\{|\eta|=1\right\}}\biggl(\int_0^{\frac{R}{t-r}}\frac{s^{n-1}}{(s^2+(1-s\omega)^2)^{\frac{n}{2}}}\omega^{-} N \, ds\biggr)d\sigma(\eta)
\end{equation}
so that
\begin{equation}\label{e: split ialfa2}
I_R^{(2)}(\alpha)=\left(I_R^{(2)}(\alpha)\right)^{+}-\left(I_R^{(2)}(\alpha)\right)^{-}.
\end{equation}
By the definition of $\omega$ in \eqref{e: definizione omega}, it follows that 
$$ \left|y\right|\omega^{-}(y)=(\psi(y)-\psi(0))^{-}=
\left\{
     \begin{array}{ll} 
		      0 & \mbox{if } \psi(y)\geq\psi(0) \\
					\psi(0)-\psi(y) & \mbox{if } \psi(y)\leq\psi(0)
		 \end{array}
.
\right. $$
We consider the case in which $\psi(y)\leq\psi(0)$. We use the fact that the graph of $\psi$ is outside the ball and so $\psi(y)\geq\sqrt{r^2-|y|^2}$. In this way we get $$0\leq\psi(0)-\psi(y)\leq r-\sqrt{r^2-\left|y\right|^2}=\frac{|y|^2}{r+\sqrt{r^2-\left|y\right|^2}}\leq\frac{|y|^2}{r}.$$
Since this estimate clearly holds in the case in which $\psi(y)\geq\psi(0)$ we get 
\begin{equation}\label{e: maggiorazione omega-}
0\leq\omega^{-}(y)=\frac{(\psi(y)-\psi(0))^{-}}{|y|}\leq\frac{|y|}{r}.
\end{equation}
Recalling the definition of $\left(I_R^{(2)}(\alpha)\right)^{-}$ in \eqref{e: ialfa2-} and using \eqref{e: gradiente limitato} and \eqref{e: maggiorazione omega-} where $y=(t-r)s\eta$ with $|\eta|=1$, we have 
\begin{align*}
0&\leq\left(I_R^{(2)}(\alpha)\right)^{-} \nonumber \\
&\leq \frac{2t}{n\omega_n r} \int_{\left\{|\eta|=1\right\}}\biggl(\int_0^{\frac{R}{t-r}}\frac{s^{n-1}}{\left(s^2+(1-s\omega)^2\right)^{\frac{n}{2}}} \frac{|s(t-r)\eta|}{r}N \, ds\biggr)d\sigma(\eta) \\ \nonumber
&\leq \frac{2t}{n\omega_n r^2}\int_{\left\{|\eta|=1\right\}}\biggl(\int_0^{\frac{R}{t-r}}\frac{s^{n-1}}{s^n}(t-r)s\sqrt{1+\left\|\nabla\psi\right\|^2_{L^{\infty}(\widehat{B}_R(0))}} \, ds\biggr)d\sigma(\eta) \nonumber \\
&\leq \frac{2b}{n\omega_n r^2}\sigma_{n-1}\sqrt{1+\left\|\nabla\psi\right\|^2_{L^{\infty}(\widehat{B}_R(0))}}	\int_0^{\frac{R}{t-r}}(t-r) \, ds \nonumber \\
&\leq cR,
\end{align*}
where $c$ is a constant depending only on $n$,$r$,$b$ and $\left\|\nabla\psi\right\|_{L^{\infty}(\widehat{B}_R(0))}$. \\ Using this estimate, together with the fact that $\left(I_R^{(2)}(\alpha)\right)^{+}$ defined in \eqref{e: ialfa2+} is positive, in \eqref{e: split ialfa2} we end up with 
\begin{equation}\label{e: maggiorazione ialfa2}
I_R^{(2)}(\alpha)\geq\left(I_R^{(2)}(\alpha)\right)^{+}-cR\geq-cR.
\end{equation}

\medbreak

\textsc{\textbf{Step 6.}} In this step we get an estimate for the limit as $\alpha\downarrow z$ of $I_R(\alpha)$. \\
Recalling the splitting of $I_R(\alpha)$ in \eqref{e: splitting I_R(alpha)} and combining it with the estimate for $I_R^{(2)}(\alpha)$ in \eqref{e: maggiorazione ialfa2} we get $$I_R(\alpha)\leq I_R^{(1)}(\alpha)+cR.$$
We compute the limit recalling \eqref{e: limite I_R(alpha)^1} and we obtain that
$$\lim_{\pedfrac{\alpha\downarrow z}{\alpha\notin\overline{\Omega}}} I_R(\alpha)\leq 1+cR.$$
This estimates holds for all $R>0$ sufficiently small and so sending $R$ to $0$, we get
\begin{equation}\label{e: limite con alpha=(0,t)}
\liminf_{R\searrow0}\biggl(\lim_{\pedfrac{\alpha\downarrow z}{\alpha\notin\overline{\Omega}}} I_R(\alpha)\biggr)\leq1.
\end{equation}

\medbreak

\textsc{\textbf{Step 7.}} In this step we conclude the proof. \\
We have proved in \eqref{e: limite con alpha=(0,t)} that if we consider $\alpha=(0,t)$ (i.e. if we consider $\alpha$ that converges orthogonally to $z$), then the limit in the definition of spherical flatness index is less or equal than $1$. \\ Since in the definition of spherical index we are taking the inferior limit over all the possible $\alpha\notin\overline{\Omega}$ and we have already proved that when $\alpha=(0,t)\notin\overline{\Omega}$ the limit is less or equal than $1$, we conclude that $$S_z(\partial \Omega,0)\leq1.$$ 
\end{proof}

\section{Proof of the main results}
In this section we prove the main results of this paper. We start by proving the stability inequality of Theorem \ref{t: relazione KG e S_z}.

\begin{proof}[Proof of Theorem \ref{t: relazione KG e S_z}]
Since the Kuran gap and the spherical index are both invariant with respect to translations, without lost of generality we can give the proof in the case $x_0=0$. \\
Recalling the definitions of Kuran gap and of $\alpha$-Kuran function we immediately have
\begin{align*}
\mathcal{K}(\partial \Omega,0)&=\sup_{\alpha\notin\overline{\Omega}}\left|k_{\alpha}(0)-\fint_{\partial \Omega} k_{\alpha} \, d\sigma\right| \nonumber \\
&\geq \left|k_{\alpha}(0)-\fint_{\partial \Omega} k_{\alpha} \, d\sigma\right| \nonumber \\
&\stackrel{k_{\alpha}(0)=0}{\geq} \frac{1}{|\partial \Omega|} \int_{\partial \Omega}k_{\alpha} \, d\sigma \nonumber \\
&\stackrel{k_{\alpha}=1+h_{\alpha}}{=} 1+\frac{1}{|\partial\Omega|}\int_{\partial\Omega}h_{\alpha} \, d\sigma.
\end{align*}
So, letting $$I_{\alpha}:=\int_{\partial \Omega} h_{\alpha} \, d\sigma$$ we have that 
\begin{equation}\label{e: maggiorazione KG}
\mathcal{K}(\partial \Omega,0)\geq 1+\frac{1}{|\partial \Omega|} I_{\alpha} \quad \forall\alpha\notin\overline{\Omega}.
\end{equation}
In order to study the behavior of $I_{\alpha}$ as $\alpha\to z$, we take $R>0$ and we consider the following splitting
\begin{align*}
I_{\alpha}&=|\alpha|^{n-2}\int_{\partial \Omega}\frac{|x|^2-|\alpha|^2}{|x-\alpha|^n} \, d\sigma(x) \nonumber \\
&= |\alpha|^{n-2}\biggl(\int_{\partial \Omega \backslash B(z,R)}\frac{|x|^2-|\alpha|^2}{|x-\alpha|^n} \, d\sigma(x) + \int_{\partial \Omega \cap B(z,R)}\frac{|x|^2-|\alpha|^2}{|x-\alpha|^n} \, d\sigma(x)\biggr).
\end{align*}
We give the following definitions
\begin{equation}\label{e: def I_alpha 1}
I_{\alpha}^{(1)}:=|\alpha|^{n-2}\int_{\partial \Omega\backslash B(z,R)}\frac{|x|^2-|\alpha|^2}{|x-\alpha|^n} \, d\sigma(x)
\end{equation}
and
\begin{equation}\label{e: def I_alpha 2}
I_{\alpha}^{(2)}:=|\alpha|^{n-2}\int_{\partial \Omega\cap B(z,R)} \frac{|x|^2-|\alpha|^2}{|x-\alpha|^n} \, d\sigma(x).
\end{equation}
In this way we have that
\begin{equation}\label{e: splitting I_alpha}
I_{\alpha}=I_{\alpha}^{(1)}+I_{\alpha}^{(2)}.
\end{equation}
Since the function $(x,\alpha)\mapsto h_{\alpha}(x)$ is smooth in $(\mathbb{R}^n\times\mathbb{R}^n)\backslash\left\{x=\alpha\right\}$ we have $$\sup_{\partial \Omega\backslash B(z,R)}|h_{\alpha}|<C(R)$$ for a constant $C(R)$ independent from $\alpha$ if $\alpha$ is sufficiently close to $z$. Therefore, we can apply the dominated convergence theorem getting $$\lim_{\pedfrac{\alpha\to z}{\alpha\notin\overline{\Omega}}}\int_{\partial \Omega\backslash B(z,R)} h_{\alpha}(x) \, d\sigma(x)=\int_{\partial \Omega\backslash B(z,R)} \lim_{\pedfrac{\alpha\to z}{\alpha\notin\overline{\Omega}}} h_{\alpha}(x) \, d\sigma(x)=\int_{\partial \Omega\backslash B(z,R)} h_z(x) \, d\sigma(x).$$
We notice that if $x\in\partial \Omega\backslash B(z,R)$ then $x\notin B(0,r)$ and $x\neq z$. So we have that $|x|\geq|z|$ and consequently $$h_z(x)=|z|^{n-2}\frac{|x|^2-|z|^2}{|x-z|^n}\geq0.$$
Then, recalling \eqref{e: def I_alpha 1}, we get 
\begin{equation}\label{e: limite I_alpha 1}
\lim_{\pedfrac{\alpha\to z}{\alpha\notin\overline{\Omega}}} I_{\alpha}^{(1)}\geq0.
\end{equation}
Keeping in mind \eqref{e: splitting I_alpha} and \eqref{e: limite I_alpha 1}, from \eqref{e: maggiorazione KG} we get 
\begin{equation}\label{e: maggiorazione2 KG}
\mathcal{K}(\partial \Omega,0)\geq 1+\frac{1}{|\partial \Omega|}\limsup_{\pedfrac{\alpha\to z}{\alpha\notin\overline{\Omega}}} I_{\alpha}^{(2)}.
\end{equation}
We notice that $$|x|^2-|\alpha|^2=|x-\alpha+\alpha|^2-|\alpha|^2=|x-\alpha|^2+2\langle x-\alpha,\alpha\rangle$$ and using this fact in \eqref{e: def I_alpha 2}, we get
\begin{equation}\label{e: splitting I_alpha 2}
I_{\alpha}^{(2)}=|\alpha|^{n-2}\biggl(\int_{\partial \Omega\cap B(z,R)}\frac{1}{|x-\alpha|^{n-2}} \, d\sigma(x)+\int_{\partial \Omega \cap B(z,R)}2\frac{\langle x-\alpha, \alpha \rangle}{|x-\alpha|^n} \, d\sigma(x)\biggr)
\end{equation}
We notice that the first integral of \eqref{e: splitting I_alpha 2} is clearly not negative and so we get 
\begin{equation}
I_{\alpha}^{(2)}\geq -2|\alpha|^{n-2}\int_{\partial \Omega\cap B(z,R)} \frac{\langle \alpha-x,\alpha \rangle}{|x-\alpha|^n} \, d\sigma(x).
\end{equation}
Therefore, keeping in mind that when $\alpha\to z$ then $|\alpha|\to r$, we have 
\begin{equation}\label{e: limsup I_alpha 2}
\limsup_{\pedfrac{\alpha\to z}{\alpha\notin\overline{\Omega}}} I_{\alpha}^{(2)}\geq -2r^{n-2}\liminf_{\pedfrac{\alpha\to z}{\alpha\notin\overline{\Omega}}}\int_{\partial \Omega \cap B(z,R)} \frac{\langle \alpha-x,\alpha \rangle}{|x-\alpha|^n} \, d\sigma(x).
\end{equation}
Combining \eqref{e: maggiorazione2 KG} and \eqref{e: limsup I_alpha 2} we get $$\mathcal{K}(\partial \Omega,0)\geq 1-\frac{2r^{n-2}}{|\partial \Omega|}\liminf_{\pedfrac{\alpha\to z}{\alpha\notin\overline{\Omega}}}\int_{\partial \Omega\cap B(z,R)} \frac{\langle \alpha-x, \alpha \rangle}{|x-\alpha|^n} \, d\sigma(x).$$ 
Since this holds for every $R>0$ sufficiently small we get 
\begin{align*}
\mathcal{K}(\partial \Omega,0)&\geq\limsup_{R\searrow0}\biggl(1-\frac{2r^{n-2}}{|\partial \Omega|}\liminf_{\pedfrac{\alpha\to z}{\alpha\notin\overline{\Omega}}}\int_{\partial \Omega\cap B(z,R)}\frac{\langle \alpha-x, \alpha \rangle}{|x-\alpha|^n} \, d\sigma(x)\biggr) \nonumber \\
&=1-\frac{n\omega_n r^{n-1}}{|\partial \Omega|}\frac{2}{n\omega_n r}\liminf_{R\searrow0}\biggl(\liminf_{\pedfrac{\alpha\to z}{\alpha\notin\overline{\Omega}}}\int_{\partial \Omega\cap B(z,R)} \frac{\langle \alpha-x,\alpha\rangle}{|x-\alpha|^n} \, d\sigma(x)\biggr) \nonumber \\
&=1-\frac{n\omega_n r^{n-1}}{|\partial \Omega|}S_z(\partial \Omega,0) \nonumber \\
&= \frac{|\partial \Omega|-S_z(\partial \Omega,0)|\partial B|}{|\partial \Omega|}
\end{align*}
and this proves \eqref{e: relazione KG e S_z}.
\end{proof}

We are now ready to prove all the corollaries.

\begin{proof}[Proof of Corollary \ref{c: condizione necessaria pseudosfere}]
Since $\partial \Omega$ is a harmonic pseudosphere centered at $x_0$, we know that $$\mathcal{K}(\partial \Omega,x_0)=0.$$ Thus, from \eqref{e: relazione KG e S_z} we get $$S_z(\partial \Omega,0)\geq\frac{|\partial \Omega|}{|\partial B|}\geq 1 \quad \forall z\in T(\partial \Omega,x_0),$$ where in the last inequality we have used the isoperimetric inequality. 
\end{proof}

\begin{proof}[Proof of Corollary \ref{c: indice S e disuguaglianza}] It immediately follows from Theorem \ref{t: relazione KG e S_z}.
\end{proof}

\begin{proof}[Proof of Corollary \ref{c: condizione sufficiente pseudosfere}] 
Clearly, since we are in the same assumptions of Corollary \ref{c: indice S e disuguaglianza}, we have that \eqref{e: maggiorazione Kuran Gap} holds. \\
Since $B\subseteq \Omega$ by the isoperimetric inequality we have 
\begin{equation}\label{e: magg Isoperimetrica}
|\partial \Omega|-|\partial B|\geq n\omega_n^{\frac{1}{n}}\left(\left|\Omega\right|^{\frac{n-1}{n}}-\left|B\right|^{\frac{n-1}{n}}\right).
\end{equation} 
Applying Lagrange mean value to the function $t\mapsto t^{\frac{n-1}{n}}$  we get 
\begin{equation}\label{e: magg Lagrange}
\left|\Omega\right|^{\frac{n-1}{n}}-\left|B\right|^{\frac{n-1}{n}}\geq \frac{n-1}{n}\frac{|\Omega\backslash B|}{\left|\Omega\right|^{\frac{1}{n}}}.
\end{equation}
Therefore, combining \eqref{e: magg Isoperimetrica} and \eqref{e: magg Lagrange}, we obtain $$|\partial \Omega|-|\partial B|\geq(n-1)\omega_n^{\frac{1}{n}}\frac{|\Omega\backslash B|}{\left|\Omega\right|^{\frac{1}{n}}}.$$ 
This inequality, together with \eqref{e: maggiorazione Kuran Gap}, leads us to the conclude that
\begin{equation}\label{e: stabilità isoperimetrica}
\mathcal{K}(\partial \Omega, x_0)\geq\frac{(n-1)\omega_n^{\frac{1}{n}}}{\left|\Omega\right|^{\frac{1}{n}}}\frac{|\Omega\backslash B|}{|\partial \Omega|}.
\end{equation}
Now, since $\partial\Omega$ is a harmonic pseudosphere, we have that $\mathcal{K}(\partial \Omega,x_0)=0$. Hence, by \eqref{e: stabilità isoperimetrica} it follows that $$\left|\Omega\backslash B\right|=0.$$ Since $\Omega$ is open and $B$ is the biggest ball centered at $x_0$ and contained in $\Omega$, we conclude that $$\Omega=B.$$
\end{proof}

\begin{proof}[Proof of Corollary \ref{c: pseudosfere lipschitz piatte}]
It immediately follows, combining Proposition \ref{p: indice sferico nel caso lipschitz piatto} and Corollary \ref{c: condizione sufficiente pseudosfere}.
\end{proof}

\section*{Appendix}
In this Appendix we prove \eqref{e: integrale notevole}, because, even though we believe that this result has been already proved, we are not able to provide an explicit reference. 

Let $\Omega$ be the Euclidean ball centered at $x_0=(0,-1)\in\mathbb{R}^{n-1}\times\mathbb{R}$ with radius $1$. We want to compute $$I:=\int_{\partial\Omega} \frac{x_n}{|x|^n} \, d\sigma(x).$$ 
First of all we notice that in $\mathbb{R}^n\backslash\{0\}$ we have that 
\begin{equation}\label{e: calcolo divergenza}
\mathrm{div}\left(\frac{x}{|x|^n}\right)=\frac{1}{2-n}\mathrm{div}\left(\nabla |x|^{2-n}\right)=0.
\end{equation}
We define the function $$F(x):=\frac{x}{|x|^n}, \quad x\neq0.$$
We notice that if $\nu$ is the outer unit normal to $B$ at the points of $\partial B$, then it is given by $$\frac{\nabla\left(x_1^2+\ldots+x_{n-1}^2+(x_n+1)^2\right)}{\left|\nabla\left(x_1^2+\ldots+x_{n-1}^2+(x_n+1)^2\right)\right|}=x+e_n.$$
Thus, for every $x\in\mathbb{R}^n\backslash\{0\}$ we have that 
\begin{equation}\label{e: prodotto scalare}
\langle F,\nu\rangle =\langle \frac{x}{|x|^n},x+e_n\rangle=\frac{1}{|x|^{n-2}}+\frac{x_n}{|x|^n}.
\end{equation}
For $\varepsilon<<1$ we define $$\Omega_{\varepsilon}:=\Omega\backslash B(0,\varepsilon).$$ We also define $$\gamma_{\varepsilon}:=\partial\Omega\cap B(0,\varepsilon), \quad \Gamma_{\varepsilon}:=\partial B(0,\varepsilon)\cap\Omega.$$ In this way, we have that
\begin{equation}\label{e: frontiera omega_epsilon}
\partial\Omega_{\varepsilon}=\left(\partial\Omega\backslash\gamma_{\varepsilon}\right)\cup\Gamma_{\varepsilon}.
\end{equation}
Thus, by the divergence theorem we have that $$\int_{\Omega_{\varepsilon}} \mathrm{div}F \, dx = \int_{\partial\Omega\backslash\gamma_{\varepsilon}} \langle F,\nu \rangle \, d\sigma -\int_{\Gamma_{\varepsilon}} \langle F,\nu \rangle \, d\sigma.$$ 
By \eqref{e: calcolo divergenza} we get 
\begin{equation}\label{e: uguaglianza integrali di superficie}
\int_{\partial\Omega\backslash\gamma_{\varepsilon}} \langle F,\nu \rangle \, d\sigma = \int_{\Gamma_{\varepsilon}} \langle F,\nu \rangle \, d\sigma,
\end{equation}
where in the last integral $\nu$ denotes the external normal to $B(0,\varepsilon)$. On $\partial B(0,\varepsilon)$ we have that $$\nu(x)=\frac{x}{|x|}=\frac{x}{\varepsilon}.$$ 
Thus, we have that 
\begin{align*}
\int_{\Gamma_{\varepsilon}} \langle F,\nu \rangle \, d\sigma &= \int_{\Gamma_{\varepsilon}} \langle \frac{x}{|x|^n},\frac{x}{|x|} \rangle \, d\sigma(x) \nonumber \\
&= \left(\frac{1}{\varepsilon}\right)^{n-1}|\Gamma_{\varepsilon}| \nonumber \\
&= \left(\frac{1}{\varepsilon}\right)^{n-1}\frac{1}{2}|B(0,\varepsilon)|(1+o(\varepsilon)) \nonumber \\
&= \frac{n\omega_n}{2}(1+o(\varepsilon)).
\end{align*}
Thus, we have proved that
\begin{equation}\label{e: primo integrale di superficie}
\int_{\Gamma_{\varepsilon}} \langle F,\nu \rangle \, d\sigma = \frac{n\omega_n}{2}(1+o(\varepsilon)).
\end{equation}
Now, we study $\displaystyle \int_{\partial\Omega\backslash\gamma_{\varepsilon}} \langle F,\nu \rangle \, d\sigma$. \\
By \eqref{e: prodotto scalare} it follows that 
\begin{align*}
\int_{\partial\Omega\backslash\gamma_{\varepsilon}} \langle F,\nu \rangle \, d\sigma &= \int_{\partial\Omega\backslash\gamma_{\varepsilon}} \left(\frac{1}{|x|^{n-2}}+\frac{x_n}{|x|^n}\right) d\sigma(x) \nonumber \\
&= \int_{\partial\Omega\backslash\gamma_{\varepsilon}} \frac{1}{|x|^{n-2}} \, d\sigma(x) + \int_{\partial\Omega\backslash\gamma_{\varepsilon}} \frac{x_n}{|x|^n} \, d\sigma(x).
\end{align*}
We notice that the function $x\mapsto\frac{1}{|x|^{n-2}}$ is summable on $\partial \Omega$ and $-\frac{x_n}{|x|^n}\chi_{\partial \Omega\backslash\gamma_{\varepsilon}}$ is a monotone increasing function with respect to $\varepsilon$. Thus, we can apply Lebesgue and Beppo Levi theorems to take the limit inside the two integrals. Thus, we have
$$\lim_{\varepsilon\to0}\int_{\partial\Omega\backslash\gamma_{\varepsilon}} \langle F,\nu \rangle \, d\sigma = \int_{\partial\Omega} \frac{1}{|x|^{n-2}} \, d\sigma(x) + \int_{\partial\Omega} \frac{x_n}{|x|^n} \, d\sigma(x).$$
Since $\partial\Omega$ is a unit Euclidean sphere and the function $x\mapsto|x|^{2-n}$ is harmonic and it is equal to $1$ at the point $(0,-1)\in\mathbb{R}^{n-1}\times\mathbb{R}$, we can apply Gauss mean value theorem getting $$\int_{\partial\Omega} \frac{1}{|x|^{n-2}} \, d\sigma(x) = n\omega_n.$$
Hence, we have proved that
\begin{equation}\label{e: secondo integrale di superficie}
\lim_{\varepsilon\to0}\int_{\partial\Omega\backslash\gamma_{\varepsilon}} \langle F,\nu \rangle \, d\sigma = n\omega_n+I.
\end{equation}
Thus, by \eqref{e: uguaglianza integrali di superficie} we have that $$\lim_{\varepsilon\to0}\int_{\partial\Omega\backslash\gamma_{\varepsilon}} \langle F,\nu \rangle \, d\sigma =  \lim_{\varepsilon\to0}\int_{\Gamma_{\varepsilon}} \langle F,\nu \rangle \, d\sigma$$ and combining \eqref{e: primo integrale di superficie} and \eqref{e: secondo integrale di superficie} we have that $$n\omega_n+I=\frac{n\omega_n}{2}$$ and so we conclude that $$I=-\frac{n\omega_n}{2}.$$  
\bigskip

\bigskip

\textbf{Acknowledgement} \ G. Cupini is   member of the \textit{Gruppo
	Nazionale per l'Analisi Ma\-te\-ma\-ti\-ca, la Probabilità e le loro Applicazioni
	(GNAMPA)} of the \textit{Istituto Nazionale di Alta Ma\-te\-ma\-ti\-ca (INdAM)} and he  acknow\-ledges financial support through the INdAM-GNAMPA Project CUP
E53C25002010001 - \emph{Esistenza e regolarit\`a per soluzioni di equazioni
	ellittiche e paraboliche anisotrope}.

\end{document}